\documentclass{elsarticle}
\usepackage[english]{babel}
\usepackage{amsfonts}
\usepackage{xcolor}
\usepackage{amsthm}
\usepackage{latexsym}
\usepackage{amsmath}
\usepackage{amssymb}
\usepackage{array}
\usepackage{arydshln}
\usepackage{tikz}
\usetikzlibrary{decorations.text,calc,arrows.meta}
\usepackage{enumerate}
\usepackage[a4paper]{geometry}
\usepackage{algorithm}
\usepackage[noend]{algpseudocode}

\newtheorem{theorem}{Theorem}

\newtheorem{corollary}{Corollary}
\newtheorem{lemma}{Lemma}

\newtheorem{question}{Question}
\numberwithin{equation}{section}
\newtheorem{conjecture}{Conjecture}
\theoremstyle{definition}
\newtheorem{definition}{Definition}
\newcommand{\red}[1]{\color{black}{#1}}

\begin{document}
\title{On the abelian complexity of generalized Thue-Morse sequences\corref{t1}}
\tnotetext[t1]{This work was supported by NSFC (Nos. 11701202, 11431007, 11871295), the Youth innovation talent  project  of colleges and universities in  Guangdong (No. 2016KQNCX229). Many thanks to the referees for a very careful reading and useful suggestions.}

\author[hzau]{Jin Chen\corref{cor1}}
\ead{cj@mail.hzau.edu.cn}
\cortext[cor1]{Corresponding author.}

\author[hust]{Zhi-Xiong Wen}
\ead{zhi-xiong.wen@hust.edu.cn}

\address[hzau]{College of Science, Huazhong Agricultural University, Wuhan 430070, China.}
\address[hust]{School of Mathematics and Statistics, Huazhong University of Science and Technology, Wuhan, 430074, China.}

\begin{abstract}
In this paper, we study the abelian complexity $\rho_n^{ab}(\mathbf{t}^{(k)})$ of generalized Thue-Morse sequences $\mathbf{t}^{(k)}$ {\red for every integer $k\geq 2.$} We obtain the exact value of $\rho_n^{ab}(\mathbf{t}^{(k)})$ for every integer $n\geq k$. Consequently,  $\rho_n^{ab}(\mathbf{t}^{(k)})$ is ultimately periodic with period $k$. Moreover, we show that the abelian complexities of a class of infinite sequences are {\red automatic sequences}.
\end{abstract}


\maketitle

\section{Introduction}
{Recently the study of the abelian complexity of infinite words was initiated by G. Richomme, K. Saari, and L. Q. Zamboni \cite{RSZ11}. For example, the abelian complexity functions of some notable sequences, such as the Thue-Morse sequence and all Sturmian sequences, were studied in \cite{RSZ11} and \cite{CH73} respectively. There are also many other works including the unbounded abelian complexity, see \cite{BBT11,CLW18,LCWW17,MR13,RSZ10} and references therein.  

{\red Fixed an integer  $k\geq 2$}. Let $\sigma_k$ be the morphism $0\mapsto 01\cdots (k-1), 1\mapsto 12\cdots (k-1){\red 0}, \cdots, k-1 \mapsto (k-1)0\cdots (k-2)$ on $\{0,1,\cdots,k-1\}$ and $\mathbf{t}^{(k)}:=\sigma_k^{\infty}(0)$. The infinite sequence $\mathbf{t}^{(k)}$ is a generalized Thue-Morse sequence with respect to $k$. Trivially, $\mathbf{t}^{(2)}$ is the infamous Thue-Morse sequence. Further, $\mathbf{t}^{(k)}$ is $k$-automatic  and uniformly recurrent (see \cite{F02}).  Recall that a sequence $\mathbf{w}=w_0w_1w_2\cdots$ is a \emph{$k$-automatic} sequence if its \emph{$k$-kernel}
$\{ (w_{{k^e}n + c})_{n \ge 0}~|~ {e \ge 0,0 \le c < k^e}\} $
{\red is} finite. If the $\mathbb{Z}$-module generated by its $k$-kernel is
finitely generated, then $\mathbf{w}$ is a \emph{$k$-regular} sequence. 

Adamczewski \cite{A03} obtained sufficient and necessary {\red conditions  for} bounded abelian complexity. There is {\red the} natrual question {\red of} finding some optimal condition for {\red the} periodic or automatic  abelian complexity.  The abelian complexity of $\mathbf{t}^{(3)}$ has been studied by Kabor\'{e} and Kient\'{e}ga \cite{KK2017}, in which they show that $\rho_n^{ab}({\mathbf{t}^{(3)}})$ is ultimately periodic with period $3$. In this paper, we are interested in the abelian complexity function
$\rho_n^{ab}(\mathbf{t}^{(k)})$ for every $n\geq k \geq 2$. The explicit value is obtained in the following theorem. 
\begin{theorem}\label{thm:main1}
For all integer $n\geq k \geq 2$, {\red{if}} $n\equiv r \mod k,$ {\red then} we have

\[\rho_n^{ab}(\mathbf{t}^{(k)}) =
\begin{cases} 
\frac{1}{4}k(k^2-1)+1  & \text{ if $k$ is odd and $r= 0$},  \\
\frac{1}{4}k(k-1)^2+k  & \text{ if $k$ is odd and $r \neq 0$},  \\
\frac{1}{4}k^3+1  & \text{ if $k$ is even and $r= 0$},  \\
\frac{1}{4}k(k-1)^2+\frac{5}{4}k  & \text{ if $k$ is even and $r\neq 0$ is even},  \\
\frac{1}{4}k^2(k-2)+k  & \text{ if $k$ is even and $r\neq 0$ is odd}. 
\end{cases} \] 
\end{theorem}
Consequently, we have the following corollary.
\begin{corollary}\label{cor:1}
$\{\rho_n^{ab}(\mathbf{t}^{(k)})\}_{n \geq 1}$ is ultimately periodic with  period $k$.
\end{corollary}
{\red In fact, we may generalize Corollary \ref{cor:1} as described below in Theorem \ref{thm:main2}, without the need to compute the exact value of the abelian complexity.}
Let $\mathcal{A}_1$ and $\mathcal{A}_2$ be two alphabets. Given any  uniformly recurrent  $k$-automatic sequence $\mathbf{w} \in {\mathcal{A}_1}^\mathbb{N}$, define a projection $\pi:{\mathcal{A}_1}^{*} \mapsto {\mathcal{A}_2}^{*}$ satisfying that 
\[ \text{for every pair of letters }a,b \in {\mathcal{A}_1}, \pi(a) \sim_{ab} \pi(b)\]
{\red where $\sim_{ab}$ will be defined in Section $2.$} Now we have the following theorem.

\begin{theorem}\label{thm:main2}
Fore every infinite sequence $\mathbf{w} \in {\mathcal{A}_1}^\mathbb{N}$ satisfying that {\red the corresponding boundary sequence} $\{\partial \mathcal{F}_n(\mathbf{w})\}_{n\geq 2}$ {\red (defined in Section $2$)} is $k$-automatic, let $\mathbf{u}=\pi(\mathbf{w}) \in {\mathcal{A}_2}^{\mathbb{N}}$ {\red be defined as prescribed above}, then the abelian complexity $\{\rho_n^{ab}(\mathbf{u})\}_{n\geq 1}$ is a $k$-automatic sequence.

\end{theorem}
This paper is organized as follows. In Section 2, we give some preliminaries and notations. In Section 3, we prove Theorem \ref{thm:main1}. In Section 4, we prove Theorem \ref{thm:main2}. 

\section{Preliminaries}
An \emph{alphabet} $\mathcal{A}$ is a finite and non-empty set (of symbols) whose elements are called \emph{letters}. A (finite)
\emph{word} over the  alphabet $\mathcal{A}$ is a concatenation of letters in $\mathcal{A}$. The concatenation of two words ${u} =
u_{0}u_{1} \cdots {\red u_m}$ and ${v} = v_{0}v_{1} \cdots v_{n}$ is the word ${uv} = u_{0}u_{1} \cdots u_{m}v_{0}v_{1} \cdots v_{n}$. The set of all finite
words over $\mathcal{A}$ including the \emph{empty word} $\varepsilon $ is denoted by $\mathcal{A}^*$. An infinite word $\mathbf{w}$ is an
infinite sequence of letters in $\mathcal{A}$. The set of all infinite words over $\mathcal{A}$ is denoted by $\mathcal{A}^{\mathbb{N}}$. Let $\Sigma_k$ be $\{0,1,\cdots, k-1\}$ for every $k\geq 1.$

The \emph{length} of a finite word ${w}\in \mathcal{A^*}$, denoted by $|w|$, is the number of letters contained in $w$. We set $\left|
\varepsilon  \right| = 0$. For any word $u\in\mathcal{A}^{*}$ and any letter $a \in \mathcal{A}$, let $|{u}|_a$ denote the number of occurrences
of $a$ in ${u}$.

A word $w$ is a factor of a finite (or an infinite) word $v$, written by $w\prec v$ if there exist a finite word $x$ and a finite (or an
infinite) word $y$ such that $v=xwy$. When $x=\varepsilon$,  ${w}$ is called a \emph{prefix} of ${v}$, denoted by ${w} \triangleleft {v}$; when
$y=\varepsilon$, $ w$ is called a suffix of ${v}$, denoted by ${w} \triangleright {v}$. For a finite word $u=u_0\cdots u_{n-1},$ denote by $u[i,j]$ the factor of $u$ from the position $i$ to $j$, i.e., $u[i,j]=u_{i-1}\cdots u_{j-1}.$ Fixed {\red an} integer $k\geq 2$, let $u \mod k:=(u_0\mod k)(u_1\mod k)\cdots (u_{n-1}\mod k).$

For a real number $x$, let $\lfloor x \rfloor$ (resp. $\lceil x \rceil$)  be the greatest {\red(resp. smallest)} integer that is less (resp. larger) than  or equal to  $x.$   
\subsection{Abelian complexity}
Given an alphabet $\mathcal{A}$. Let $\mathbf{w}=w_{0}w_{1}w_{2}\cdots \in \mathcal{A}^{\mathbb{N}}$ be an infinite word. Denote by ${\mathcal{F}_{n}(\mathbf{w})}$ the set of all
factors of $\mathbf{w}$ of length $n$, i.e., \[{\mathcal{F}_{n}(\mathbf{w})}: = \{w_{i}w_{i + 1}\cdots w_{i + n - 1} : i \geq 0 \}.\]
In fact, when $\mathbf{w}$ is a finite word, ${\mathcal{F}_{n}(\mathbf{w})}$ is still well defined. Write $\mathcal{F}_{\mathbf{w}}:=\cup_{n\geq 1}{\mathcal{F}_{n}(\mathbf{w})}$.  The \emph{subword complexity function} $\rho_{n}(\mathbf{w})$ of $\mathbf{w}$ is defined by
\[\rho_{n}(\mathbf{w}) := \# \mathcal{F}_{n }(\mathbf{w}).\]

Fixed the alphabet $\mathcal{A}=\{a_0,\cdots,a_{q-1}\}$ and a factor $v$ of an infinite word $\mathbf{w}$, the Parikh vector of $v$ is the $q$-uplet \[\psi({\red v}) := (|v|_{a_0} , |v|_{a_1} , \cdots, |v|_{a_{q-1}} ).\] Denote by  ${\red \Psi_{n}(\mathbf{w})}$, the set of the Parikh vectors of the factors of length $n$ of ${\mathbf{\red w}}$:
\[ \Psi_{n}( \mathbf{w}) := \{\psi(v) : v \in {\mathcal{F}_{n}(\mathbf{w})}\}. \]

The \emph{abelian complexity function} of ${\mathbf{w}}$ is defined by 
\[ \rho^{ab}_{n}( \mathbf{w}) := \#\Psi_{n}( \mathbf{w}). \]
In fact, we could give {\red another} definition of the abelian complexity function induced by the abelian equivalence relation.  For the infinite word $\mathbf{w}$, given two factors $u,v$ we say $u$ is \emph{abelian  equivalent} to $v$ (write $u\sim_{ab} v$) if $\psi(u)=\psi(v).$ Now the abelian complexity function of $\mathbf{w}$ can be defined as follows:
\[\rho_{n}^{ab}(\mathbf{w}) := \# \{ \mathcal{F}_{n }(\mathbf{w})/{\sim_{ab}} \}.\]

In addition, before the proof of the main theorems, we give some notations which will be important through all the paper. Let $\mathbf{w} \in \mathcal{A}^{\mathbb{N}}$ be an infinite word, given a factor $u=u_0\cdots u_{n-1}$, denote by $\partial u$ the {\red{\emph{$1$-length boundary word}  consisting}} of the first and last letter of $u,$
\[ \partial u = u_0u_{n-1}.\]
For the {\red sake of} completeness, let $\partial u$ be itself $u$ when $|u|\leq 1.$ Define {\red the set of the \emph{boundary words}} of factors with length $n$ of $\mathbf{w}$ by
\[ \partial\mathcal{F}_{n}(\mathbf{w}) :=  \{\partial u : u \in {\mathcal{F}_{n}(\mathbf{w})} \}. \]

\section{Abelian complexity of $\mathbf{t}^{(k)}$ }
First we give some notations and lemmas. For every $n\geq k \geq 2,$ set $n=km+r$ for some $m\geq 1$ and $r=0,\cdots k-1.$ Let
\[ S_n(\mathbf{w}) := \bigcup_{ab \in \partial\mathcal{F}_{m+1}(\mathbf{w}) } \mathcal{F}_{r+k}\big(\sigma_k(ab)\big)  \] 

and
\[ G_n(\mathbf{w}) := \bigcup_{ab \in \partial\mathcal{F}_{m+2}(\mathbf{w}) } \mathcal{F}_{r}\big(\sigma_k(ab)[k-r+2,k+r-1]\big)  \] 
\begin{lemma}\label{lem:toboundary}
For every $n\geq k \geq 2,$ set $n=km+r$ for some $m\geq 1$ and $r=0,\cdots k-1.$  
If $r\leq 1,$ then we have
\[ \rho_{n}^{ab}(\mathbf{t}^{(k)}) =  \#\{ \psi(u): u\in S_n(\mathbf{t}^{(k)}) \}. \]
Otherwise{\red,} if $r\geq 2,$ then we have
\[ \rho_{n}^{ab}(\mathbf{t}^{(k)}) =  \#\{ \psi(u) \cup (\mathbb{I}_k+\psi(v)): u\in S_n(\mathbf{t}^{(k)}), v \in G_n(\mathbf{t}^{(k)}) \} \]
where $\mathbb{I}_k = (1,1,\dots,1)$ is the vector of length $k$ with all the element being the same $1.$		
\end{lemma} 

\begin{proof}
For every $n\geq k \geq 2,$ set $n=km+r$ for some $m\geq 1$ and $r=0,\cdots k-1.$ Given any factor $u$ of length $n$ of $\mathbf{t}^{(k)},$ there exists a factorization (maybe not unique) such that 
\[ u=\alpha \sigma_k(v) \beta \]
for some factors $\alpha, v,\beta$ satisfying that $|v|=m-1, ~|\alpha|+|\beta| = k+r.$
{\red We can assume that, if $r\neq 0$, then both $\alpha$ and $\beta$ are not the empty word.}  In fact, if $r>0$ and $\alpha=\varepsilon$, then $|\beta|=k+r$ and let $\beta={\red \sigma_k}(a)\beta^{\prime}$ for some letter $a$. Set $v=v_0\cdots v_{m-2},$ now we cloud set the new prefix $\alpha^{\prime}=\sigma(v_0)$ and $v^{\prime}=v_1\cdots v_{m-2}a$, then
\[ u=\alpha^{\prime} \sigma_k(v^{\prime})\beta^{\prime}. \]
 It is {\red trivial} that \[ \psi(u) = \psi(\alpha) +\psi(\sigma_k(v))+\psi( \beta)=\psi(\alpha\beta)+(m-1)\mathbb{I}_{\red k}. \]
 
$\mathbf{(1)}$ If $r\leq 1,$  the length of {\red the} preimage of $\alpha$ and $\beta$ are both $1$  except {\red in the case that $\alpha$  or $\beta$ is the empty word}. In other words, $\alpha \triangleright \sigma_k(a),~\beta \triangleleft \sigma_k(b)$ for some letters $a,b\in \{0,\cdots, k-1\}.$ Consequently,  $\alpha\beta$ is a factor of length $k+r$ of  $\sigma(ab)$ where $ab \in \partial{\mathcal{F}_{m+1}(\mathbf{t}^{(k)})}.$ At the {\red same} time, it is easy to know that all the possible $\alpha\beta$ is exactly the factor set of length $k+r$ of ${\red \sigma_k}(ab)$ for the above $ab$ even if $r=0$. i.e.,
\begin{eqnarray*}
\rho^{ab}_{n}( \mathbf{t}^{(k)}) &=& \#\{\psi(u) : u \in {\mathcal{F}_{n}(\mathbf{t}^{(k)})}\} \\
						  &=& \#\{\psi(\alpha\beta): \alpha \sigma_k(v) \beta \in {\mathcal{F}_{n}(\mathbf{t}^{(k)})}\} \\
						  &=& \#\{\psi(\alpha\beta): \alpha  \beta  \prec \sigma_k(ab)\text{ with }ab \in \partial{\mathcal{F}_{m+1}(\mathbf{t}^{(k)})}, ~|\alpha\beta|=k+r \} \\
						  &=& \#\{\psi(v): v \in S_n(\mathbf{t}^{(k)}) \}
\end{eqnarray*} 
which is the desired result.

$\mathbf{(2)}$ If $r\geq 2,$ the length vector set of preimage of $\alpha$ and $\beta$ is  $\{(1,1),(2,1),(1,2)\}.$  When the length vector is  $(1,1)$, we have the same argument with the above case $r\leq 1.$ It is {\red sufficient} to consider the case {\red where} the length vector is $(1,2),$  since $(2,1)$ and $(1,2)$ are symmetric case.

Now assume that $|\beta|>k$ and $1 \leq |\alpha|<r$, and
\[ \alpha \triangleright \sigma_k(a),~\beta=\sigma_k(b)\beta^{\prime} \text{ with } \beta^{\prime} \triangleleft \sigma_k(c) \]
for some letters $a,b,c \in \{0,\cdots, k-1\}.$ {\red Note that the value of the letter $b$ does not affect the Parikh vector of $\alpha\beta,$ and that the word $ac$ is a boundary word of a factor of length $m+2$ of $\mathbf{t}^{(k)}.$} At the same time, {\red since} we have that $|\alpha|+|\beta^{\prime}|=r,$ the length of $\beta^{\prime}$ ranges from $1$ to $r-1.$ In other words,  
\[ \alpha\beta^{\prime} \prec \sigma_k(ac)[k-r+2,k+r-1]. \]
Moreover, all the possible $\alpha\beta^{\prime}$ is the set of {\red factors} of length $r$ of {\red the form} $\sigma_k(ac)[k-r+2,k+r-1].$ For the length vector $(2,1)$,  we can apply the same {\red reasoning}. Overall, we have that 
\begin{eqnarray*}
\rho^{ab}_{n}( \mathbf{t}^{(k)}) &=& \#\{\psi(u) : u \in {\mathcal{F}_{n}(\mathbf{t}^{(k)})}\} \\
								&=& \#\{\psi(v) \cup \psi(0u): v \in S_n(\mathbf{t}^{(k)}),~ u \in  G_n(\mathbf{t}^{(k)}) \} \\
								&=& \#\{\psi(v) \cup (\mathbb{I}_k+\psi(u)): v \in S_n(\mathbf{t}^{(k)}),~ u \in  G_n(\mathbf{t}^{(k)}) \} 
\end{eqnarray*} 
This {\red completes} the proof.
\end{proof}
According to Lemma \ref{lem:toboundary}, $\{\partial\mathcal{F}_{m}(\mathbf{w})\}_{m\geq 2}$ is the key point to prove Theorem \ref{thm:main1}.

\begin{lemma}\label{lem:tkboudary}
For the generalized Thue-Morse sequence $\mathbf{t}^{(k)}$ and every integer $n\geq 2,$ we have
\[ \partial\mathcal{F}_{n}(\mathbf{t}^{(k)}) ={\Sigma_k}^2.\] 
\end{lemma}
\begin{proof}
We shall prove by induction by $n.$ Obviously this lemma holds for $n=2,\cdots,k.$ Suppose it holds for all $i\leq n,$ we will show it holds for $n+1.$ 
{\red Set $n+1=mk+r$ with  $1\leq m$ and $1\leq r\leq k,$ there exists factors of length $m+1$} in the form of $aub$ where $ab$ is {\red a} non-empty word with $|ab|{\red =} 2.$ By the assumption, $ab$ could go through all the words in ${\Sigma_k}^{|ab|}.$
We know that for every letter (integer) $i \in \{0,1.\cdots,k-1\},$ 
\[ \sigma_k(i) = i (i+1) \cdots (i+k-1) \mod k.\]
Now we consider the finite word $v=\sigma_k(aub)[1,n+1],$ which begins with $a$ and ends with $( b+r-1)\mod k.$  It follows that this lemma holds for $n+1$ in this case. 

\end{proof}
For simplicity, let $|V|_i$ be the {\red number of occurrences of $i$ as an entry of the} finite dimensional vector $V.$

\begin{lemma}\label{lem:main1-req0}
For every $n\geq k \geq 2,$ if $n=km$ for some $m\geq 1,$ then  we have
\[ \rho_{n}^{ab}(\mathbf{t}^{(k)}) =  1+\sum_{t=1}^{\lfloor \frac{k}{2}\rfloor} k(k-2t+1)=\begin{cases} 
1+\frac{1}{4}k^3 & \text{ if $k$ is even,} \\
1+\frac{1}{4}k(k^2-1) & \text{ otherwise.} \\
\end{cases} \]
\end{lemma}
\begin{proof}
 By Lemma \ref{lem:toboundary}, when $r=0$, it {\red suffices} to consider $S_n(\mathbf{t}^{(k)})$. In this case, {\red it follows from Lemma \ref{lem:tkboudary} that}
\[ S_n(\mathbf{t}^{(k)}) = \bigcup_{ab \in {\Sigma_k}^2} \mathcal{F}_{k}\big(\sigma_k(ab)\big)=\bigcup_{ab \in {\Sigma_k}^2,1\leq t \leq {\red k+1}} \{\sigma_k(ab)[t,t+k-1]\}. \]
For every $u \in S_n(\mathbf{t}^{(k)}),$ $|u|_i\leq 2$ for every $i \in \Sigma_k.$ The only Parikh vector without $2$ is $\mathbb{I}_k.$ It {\red suffices} to consider the number and the position of $2$s in the elements of Parikh {\red vectors} since the numbers of $2$ and $0$ are the same. In fact, for every $u \in S_n(\mathbf{t}^{(k)}),$ write 
\[ u=\alpha\beta=\alpha_1\alpha_2\cdots \alpha_{t} \beta_1\beta_2\cdots \beta_{k-t}=(\alpha_1\cdots(\alpha_1+t-1)\beta_1\cdots (\beta_1+k-t-1) \mod k) \]
where $\alpha \triangleright \sigma(a), \beta \triangleright \sigma(b)$ for some letters $a,b.$
\[ |\psi(u)|_0+|\psi(u)|_1+|\psi(u)|_2=k\]
and 
\[ |\psi(u)|_1+2|\psi(u)|_2=k,\]
which implies $|\psi(u)|_2=|\psi(u)|_0.$ We put the $k$ elements of {\red a} Parikh vector into a circle, {\red as seen} in Figure \ref{fig:Parikh_circle}. {\red With the help of that figure}, it is not hard to know that the positions of the letters $2$ and $0$ are both a consecutive block in this circle. The length of $2$s' block ranges from $1$ to $\lfloor \frac{k}{2} \rfloor.$ For every fixed length $t$ of $2$s' block there are $k$ positions to place the  consecutive $2$s' block  of length $t$. Next we only need to place the consecutive block of $0$s of length $t$ in the left $(k-t)$ positions. Hence the number of Parikh {\red vectors} with at least one element $2$ is 
\[ \sum_{t=1}^{\lfloor \frac{k}{2}\rfloor} k(k-2t+1).\]
Adding the Parikh vector $\mathbb{I}_k,$ we complete the proof.
\end{proof}

\begin{figure}[htpb]
\begin{center}
\begin{tikzpicture}[scale=0.7]

\coordinate (O) at (0,0);

\draw[fill=white] (O) circle (2.5);
\draw[fill=white] (O) circle (1.1);
 \pgfmathsetmacro\angdiv{360/12}
 \draw
\foreach \i in {0,...,11}{
        ($({90-(\i-1/2)*\angdiv}:2.5)$) -- ($(({90-(\i-1/2)*\angdiv}:1.1)$)
    };

\foreach \i in {0,...,9}{
\node at ({105-(\i-1/2)*\angdiv}:1.8) {$\i$};
}
\node at ({105-(10-1/2)*\angdiv}:1.8) {$\cdots$};
\node at ({105-(11-1/2)*\angdiv}:1.8) {\small{$k-1$}};

\draw[thick,red] ([shift=({90-(2-1/2)*\angdiv}:3cm)]O) arc ({90-(2-1/2)*\angdiv}:{90-(8-1/2)*\angdiv}:3cm) node at ({90-(5-1/2)*\angdiv}:3.5) {$\alpha$};

\draw[thick,blue] ([shift=({90-(6-1/2)*\angdiv}:3.5cm)]O) arc ({90-(6-1/2)*\angdiv}:{90-(12-1/2)*\angdiv}:3.5cm) node at ({90-(9-1/2)*\angdiv}:4) {$\beta$};

\end{tikzpicture}
\end{center}
\vspace{-5pt}
\caption{Parikh vector circle over $\Sigma_k$}\label{fig:Parikh_circle}
\end{figure}
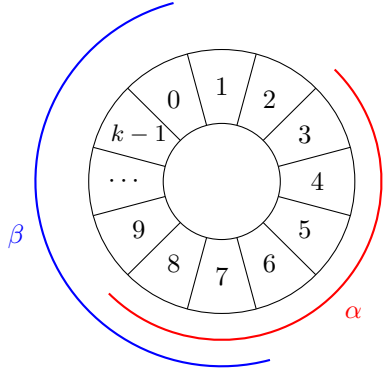

\begin{lemma}\label{lem:main1-rneq0}
For every $n\geq k \geq 2,$ if $n=km+r$ for some $m\geq 1$ and $r=1,\cdots,k-1,$ then we have
\begin{eqnarray*}
 \rho_{n}^{ab}(\mathbf{t}^{(k)}) &=&  k+\frac{1}{2}{\red k(r-1)(k-r-1)}+\sum_{{\red \tau}=1}^{\lfloor \frac{r}{2}\rfloor} k(r-2{\red\tau}+1)+\sum_{t=1+r}^{\lfloor \frac{k+r}{2}\rfloor} k(k+r-2t+1) \\
&=&\begin{cases} 
\frac{1}{4}k(k-1)^2+k  & \text{ if $k$ is odd},  \\
\frac{1}{4}k(k-1)^2+\frac{5}{4}k  & \text{ if $k$ is even and $r$ is even},  \\
\frac{1}{4}k^2(k-2)+k  & \text{ if $k$ is even and $r$ is odd}. 
\end{cases} 
\end{eqnarray*}

\end{lemma}
\begin{proof}
 {\red Taking into account Lemma \ref{lem:toboundary} and   Lemma \ref{lem:tkboudary}, we divide the proof into two subcases.}
 
 $\mathbf{(1)}$ when $r=1$, we only need {\red to} investigate $S_n(\mathbf{t}^{(k)})$.  We have that
\[ S_n(\mathbf{t}^{(k)}) = \bigcup_{ab \in {\Sigma_k}^2} \mathcal{F}_{k+1}\big(\sigma_k(ab)\big)=\bigcup_{ab \in {\Sigma_k}^2,1\leq t \leq k} \{\sigma_k(ab)[t,t+k]\}. \]
  For every $u \in S_n(\mathbf{t}^{(k)}),$ $|u|_i\leq 2$ for every $i \in \Sigma_k.$ Moreover, there is at least one element being $2$ in $\psi(u).$ After putting the $k$ elements of {\red the} Parikh vector into a circle, the position of $2$ and $0$ are both a consecutive block. If there is only one $2$ as the element in $\psi(u),$ i.e., $|\psi(u)|_2=1,$  then the possible number of {\red such Parikh vectors} is $k,$ since $a\sigma_k(0) \in S_n(\mathbf{t}^{(k)})$ for every $a \in \Sigma_k.$ Now set $|\psi(u)|_2=t,$  we need to consider the $t$ ranging from $2$ to $\lfloor (k+1)/2 \rfloor.$ At the {\red same} time, we have that $|\psi(u)|_0=t-1.$ It is {\red sufficient} to place the consecutive $0$s' block in the left $(k-t)$ positions. i.e., the number of Parikh vector having $\geq 2$ consecutive $2$s' block is 
\[ \sum_{t=2}^{\lfloor \frac{k+1}{2}\rfloor} k(k-2t+2). \]
Adding the $k$ Parikh {\red vectors} with exactly one $``2"$ as the element, the lemma holds.

$\mathbf{(2)}$ when $r>1$, $S_n(\mathbf{t}^{(k)})$ and $G_n(\mathbf{t}^{(k)})$ are both important to the result. 
\[ S_n(\mathbf{t}^{(k)}) = \bigcup_{ab \in {\Sigma_k}^2} \mathcal{F}_{k+r}\big(\sigma_k(ab)\big)=\bigcup_{ab \in {\Sigma_k}^2,1\leq t \leq k-r+1} \{\sigma_k(ab)[t,t+k+r-1]\} \]
and
\[ G_n(\mathbf{t}^{(k)}) = \bigcup_{ab \in {\Sigma_k}^2 } \mathcal{F}_{r}\big(\sigma_k(ab)[k-r+2,k+r-1]\big). \]
Following from Lemma \ref{lem:toboundary},  for every word $u$ in $G_n(\mathbf{t}^{(k)})$, when computing the abelian complexity, $\mathbb{I}_k$ should be added to  $\psi(u)$. For simplicity, we consider the following {\red set} $G^{\prime}_n(\mathbf{t}^{(k)})$ instead of  $G_n(\mathbf{t}^{(k)}).$

\[ G^{\prime}_n(\mathbf{t}^{(k)}) = \bigcup_{ab \in {\Sigma_k}^2 } \mathcal{F}_{k+r}\big(\sigma_k(acb)[k-r+2,2k+r-1]\big). \]
Note that the letter $c$ can be arbitrary in $\Sigma_k.$  Given any $u \in G^{\prime}_n(\mathbf{t}^{(k)}),$ write 
\begin{eqnarray*}
u=\alpha\sigma_k(c)\beta&=&\alpha_1\alpha_2\cdots \alpha_{t}\sigma_k(c) \beta_1\beta_2\cdots \beta_{r-t} \\
&=&(\alpha_1\cdots(\alpha_1+t-1)\sigma_k(c)\beta_1\cdots (\beta_1+r-t-1) \mod k)
\end{eqnarray*}
where $\alpha \triangleright \sigma(a), \beta \triangleright \sigma(b)$ for some letters $a,b.$ 

$\mathbf{(i)}$ First we focus on the case that the maximal element of 
{\red the} Parikh vector is $2.$ For every word $u \in G^{\prime}_n(\mathbf{t}^{(k)}),$ then the length of consecutive $2$s' block is strictly less than $r+1.$ Obviously the number of Parikh {\red vectors} with exactly $r$ length of consecutive $2$s' block as the element is $k.$ For the word $u=\alpha\sigma_k(c)\beta,$ we have \[ \{\alpha_i: 1\leq i \leq t\} \cap \{\beta_i: 1\leq i \leq r-t\} =\emptyset.\]
First we put $\alpha$ into the circle, {\red the number of the choices for that being equal to $k.$}  We can assume $\beta_1$ and $\beta_{\red r-t}$ both are not belongs to ${\red \{\alpha_{t}+1, \alpha_1-1\} \mod k.}$ i.e., the gap between $\alpha$ {\red and} $\beta$ is positive. Otherwise, $|\psi(u)|_2=r,$ which has been considered. There are $(k-t-2)$ positions for $\beta.$ Hence all the possible number of Parikh vector is
\[ \frac{1}{2}\sum_{t=1}^{r-1} k\big((k-t-2)-(r-t)+1\big)= \frac{1}{2}{\red k(r-1)(k-r-1)}.\]

We could apply the same argument of case $\mathbf{(1)}$ to obtain the number of Parikh vector having $\geq r+1$ consecutive $2$s' {\red blocks}:
\[\sum_{t=1+r}^{\lfloor \frac{k+r}{2}\rfloor} k(k+r-2t+1).\]
Overall, the total number of Parikh {\red vectors} in this case is
\[ k+\frac{1}{2}{\red k(r-1)(k-r-1)}+\sum_{t=1+r}^{\lfloor \frac{k+r}{2}\rfloor} k(k+r-2t+1).\]

$\mathbf{(ii)}$ Then {\red the case that is left to consider is} that the maximal element of the corresponding Parikh vector is $3.$ 
For the word $u=\alpha\sigma_k(c)\beta,$ we have \[ \#(\{\alpha_i: 1\leq i \leq t\} \cap \{\beta_i: 1\leq i \leq r-t\}):={\red \tau} >0.\]
Morever, $|\psi(u)|_3= {\red\tau}.$ Clearly, ${\red\tau}$ ranges from $1$ to $\lfloor \frac{r}{2}\rfloor.$ {\red Also the positions of $3$ and $1$  both form consecutive blocks}. In the same manner {\red as before}, we can obtain the  number of Parikh {\red vectors} in this case, {\red which is}
\[ \sum_{{\red\tau}=1}^{\lfloor \frac{r}{2}\rfloor} k(r-2{\red\tau}+1). \]
We know that the subcase $\mathbf{(i)}$ and subcase $\mathbf{(ii)}$ do not intersect, 
{\red which implies} that this lemma holds.
\end{proof}
\begin{proof}[Proof of Theorem \ref{thm:main1}]
It follows from Lemma \ref{lem:main1-req0} and Lemma \ref{lem:main1-rneq0} that Theorem \ref{thm:main1} holds.
\end{proof}

Richomme et al \cite{RSZ11} obtain the sufficient and necessary condition for aperiodic sequence $\mathbf{w}\in \Sigma_2^{\mathbb{N}}$ satisfying $\rho_n^{ab}(\mathbf{w})=\rho_n^{ab}(\mathbf{t}^{(2)}).$ It is natural to ask the following question.
\begin{question}
For the aperiodic sequence $\mathbf{w}\ \in \Sigma_k^{\mathbb{N}}$ with $k\geq 3,$ what is the sufficient and necessary condition for $\mathbf{w}$ satisfying $\rho_n^{ab}(w)=\rho_n^{ab}(\mathbf{t}^{(k)})$?
\end{question}

\section{Abelian complexities of a class of infinite sequences}
Let $\mathbf{w}=\{w_n\}_{n\geq 0}$ be an infinite sequence. Define the \emph{$k$-{\red kernel}} of $\mathbf{w}$ to be the set of subsequence 
\[ \mathcal{K}_k(\mathbf{w}):=\{ (w_{{k^e}n + c})_{n \ge 0}~|~ {e \ge 0,0 \le c < k^e}\} \]
\begin{definition}\label{def:automaic}
Let $k\geq 2,$ $\mathbf{w}$ is $k$-automatic if and only if $\mathcal{K}_k(\mathbf{w})$ is finite.
\end{definition}

\begin{lemma}{\cite[Theorem 6.8.2]{AS03}}\label{lem:auto-periodic}
{\red Consider a} finite alphabet $\mathcal{A}$ and ${\red a> 0}.$ Let $\{w_n\}_{n\geq 0}$ be a sequence taking values in $\mathcal{A}$ such that $\{w_{an+i}\}_{n\geq 0}$ is $k$-automatic for $0\leq i < a$. Then $\{w_n\}_{n\geq 0}$ is itself $k$-automatic.
\end{lemma}
\begin{lemma}\cite[Theorem 6.9.2]{AS03}\label{lem:closed}
Automatic sequences are closed under $1$-uniform transducers.
\end{lemma}
Following from Lemma \ref{lem:closed}, we have the following corollary which will be useful.
\begin{corollary}\label{cor:concact-auto}
For every $k$-automatic sequence $\mathbf{w}=\{w_n\}_{n\geq 0}\in \mathcal{A}^{\mathbb{N}}$, $\{(w_nw_{n+1})\}_{n\geq 0} \in (\mathcal{A}\times \mathcal{A})^{\mathbb{N}}$ is also $k$-automatic.
\end{corollary}

Recall that $\mathbf{u}=\pi(\mathbf{w})$ for some  $k$-automatic sequence $\mathbf{w} \in {\mathcal{A}_1}^\mathbb{N}$ with $\pi:{\mathcal{A}_1}^{*} \mapsto {\mathcal{A}_2}^{*}$ satisfying that 
\[ \text{for every pair of letters }a,b \in {\mathcal{A}_1}, \pi(a) \sim_{ab} \pi(b)\]

For every $n\geq k \geq 2,$ set $n=km+r$ for some $m\geq 1$ and $r=0,\cdots k-1.$ Let
\[ S_n^{\prime}(\mathbf{w}) := \bigcup_{ab \in \partial\mathcal{F}_{m+1}(\mathbf{w}) } \mathcal{F}_{r+k}\big(\pi(ab)\big)  \] 

and
\[ G_n^{\prime}(\mathbf{w}) := \bigcup_{ab \in \partial\mathcal{F}_{m+2}(\mathbf{w}) } \mathcal{F}_{r}\big(\pi(ab)[k-r+2,k+r-1]\big)  \]

We shall give the following lemma without the proof, since the method is the same {\red as in} Lemma \ref{lem:toboundary},  with $\mathbf{t}^{(k)}$ being replaced by $\mathbf{w}$.
\begin{lemma}\label{lem:wtoboudary}
For every $n\geq k \geq 2,$ set $n=km+r$ for some $m\geq 1$ and $r=0,\cdots, k-1.$ If $r\leq 1,$ then we have
\[ \rho_{n}^{ab}(\mathbf{u}) =  \#\{ \psi(v): v\in S_n^{\prime}(\mathbf{w}) \}. \]
Otherwise if $r\geq 2,$ then we have
\[ \rho_{n}^{ab}(\mathbf{u}) =  \#\{ \psi(v) \cup (\psi(\pi(a)v^{
\prime})): v\in S_n^{\prime}(\mathbf{w}), v^{\prime} \in G_n^{\prime}(\mathbf{w}) \} \]
where $a$ is an arbitrary letter in $\mathcal{A}_1.$	
\end{lemma}

\begin{proof}[Proof of Theorem \ref{thm:main2}]
By Lemma \ref{lem:auto-periodic}, it suffices to show that $\{w_{km+r}\}_{m\geq 0}$ is $k$-automatic for $0\leq r < k$. If $ r\leq 1,$ then it {\red follows} from Lemma \ref{lem:wtoboudary} that for every $n=km+r$ with ${\red k\geq 2}$ and $r=0,1,$ we have
\[ \rho_{n}^{ab}(\mathbf{u}) =\rho_{km+r}^{ab}(\mathbf{u})=  \#\left\{ \psi(v): v\in \bigcup_{ab\in \partial F_{m+1}(\mathbf{w})} \mathcal{F}_{r+k}(\pi(ab)) \right\}. \]

This implies that the abelian complexity of length $km+r$ for $u$ is uniquely dependent on the set $\partial F_{m+1}(\mathbf{w}).$ Following from Lemma \ref{lem:closed},  $\{\rho_{km+r}^{ab}(\mathbf{u})\}_{m\geq 1}$ is $k$-automatic for $r=0,1$.

When $ r\geq 2,$ it follows from Lemma \ref{lem:wtoboudary} that for every $n=km+r$ with ${\red k\geq 2}$ and $r\in [2,k-1),$ we have
\begin{eqnarray*}
\rho_{km+r}^{ab}(\mathbf{u}) = \#\text{\huge$\{$}&&\{ \psi(v): v\in \bigcup_{ab\in \partial F_{m+1}(\mathbf{w})} \mathcal{F}_{r+k}(\pi(ab)) \}  \\
	&&\cup ~\{ (\psi(\pi(a)v^{
\prime})):v^{\prime} \in \bigcup_{ab \in \partial\mathcal{F}_{m+2}(\mathbf{w}) } \mathcal{F}_{r}\big(\pi(ab)[k-r+2,k+r-1]\big)\}\text{\huge$\}$}
\end{eqnarray*}
 This implies that the abelian complexity of length $km+r$ for $u$ is uniquely dependent on the set $\partial F_{m+1}(\mathbf{w})$ and  $\partial F_{m+2}(\mathbf{w}).$ By Lemma \ref{lem:closed} and Corollary \ref{cor:concact-auto},  $\{\rho_{km+r}^{ab}(\mathbf{u})\}_{m\geq 1}$ is $k$-automatic for $r\geq 2$. This completes the proof since the leading $k-1$ terms can not affect the $k$-automatic property.
\end{proof}

Following from Lemma \ref{lem:tkboudary}, $\partial \mathcal{F}_{n}(\mathbf{t}^{(k)})$ is periodic, which is a special $\ell$-automatic sequence for any $\ell\geq 2.$ This implies that  Corollary \ref{cor:1} holds by Theorem \ref{thm:main2}. {\red Since we} know that the key point for Theorem \ref{thm:main2} is the boundary {\red sequences}, it is significant to consider the infinite sequences {\red for which the corresponding boundary sequence is automatic.  Naturally we have the following Conjecture \ref{conj:1}. As far as we know, this conjecture holds} for the periodic sequences, the above $\mathbf{t}^{(k)}$ and Cantor sequence $c$ which is the fixed point of the morphism $0\mapsto 000,1\mapsto 101$.
\begin{conjecture}\label{conj:1}
For every $k$-automatic sequence $\mathbf{w}=\{w_n\}_{n\geq 0}\in \mathcal{A}^{\mathbb{N}},$ {\red the corresponding boundary sequence} $\{\partial\mathcal{F}_n(\mathbf{w})\}_{n\geq 2} \in (2^{\mathcal{A}\times \mathcal{A}})^{\mathbb{N}}$ is also $k$-automatic.
\end{conjecture}

\end{document}